\documentclass[a4paper,12pt,openany,reqno]{amsart}
 	\usepackage{amsmath}
 \usepackage{amssymb}
 \usepackage{amsbsy} 
 \usepackage{latexsym, amsfonts}
 \usepackage{graphics}
 \usepackage{ulem}
 \usepackage{hhline}
 \usepackage{dsfont}
 \usepackage{mathrsfs}
 \usepackage{color}
 \usepackage{fancyhdr}
 \usepackage{rotating}
 \usepackage{fancybox}
 \usepackage{colortbl}
 \usepackage{pifont}
 \usepackage{setspace}
 \usepackage{enumerate}
 \usepackage{multicol}
 \usepackage{varioref}
 \usepackage{lmodern}
 \usepackage{textcomp}
 \usepackage{euscript}
 \usepackage[pdftex]{hyperref}
\usepackage{latexsym}
\usepackage{enumerate} 
\usepackage{color} 

\usepackage{indentfirst} 
\usepackage{hyperref}
\hypersetup{
colorlinks=true, 
breaklinks=true, 
urlcolor= blue, 
linkcolor= blue, 
citecolor=blue, 
anchorcolor = blue, citecolor = blue, filecolor = blue}
\usepackage{setspace} 
\theoremstyle{plain} 
\newtheorem{theorem}{Theorem}[section]

\newtheorem{corollary}[theorem]{Corollary}
\newtheorem{lemma}[theorem]{Lemma}

\newtheorem{remark}[theorem]{Remark}
\newtheorem{remarks}[theorem]{Remarks}  
\usepackage[a4paper,tmargin=3 cm,bmargin=3 cm,rmargin=3 cm,lmargin=3 cm]{geometry}

 \makeatletter
 \@namedef{subjclassname@2010}{\textup{2010} Mathematics Subject Classification}
 \makeatother
 
\begin{document}

 \title[\tiny{On the completed tensor product of two preadic algebras over a field}]{On the completed tensor product of two preadic algebras over a field}
 
 	\author[M. Taba\^{a}]{Mohamed Taba\^{a}}


 	\address{DEPARTMENT OF MATHEMATICS\\
 		FACULTY OF SCIENCES\\
 		MOHAMMED V UNIVERSITY IN RABAT\\
 		RABAT\\
 		MOROCCO }
 	\email{mohamedtabaa11@gmail.com} 
 	\begin{abstract}
 		In this paper, we study the problem, initiated by Grothendieck, of the transfer of the properties of two preadic noetherian algebras over a field to their completed tensor product.
 	\end{abstract}

 	\maketitle

 	\section{Introduction}\
 	
 	All the rings are commutative with identity and all ring homomorphisms are unital. The notations
 	are those of \cite{Grothendieck2} and the terminology is that of \cite{GrothendieckDieudonne}.
 	
 	Let $k$ be a field, $A$ and $B$ two preadic noetherian $k$-algebras and $A%
 	\hat{\otimes}_{k}B$ their completed tensor product. The problem of the transfer of the properties of $A$ and $B$ to\ $A\hat{%
 		\otimes}_{k}B$ is initiated by Grothendieck in \cite[(7.5)]{Grothendieck2}, where he treated
 	the case where $A$ and $B$ are local and their residue fields are finite
 	extensions of $k$. This paper is devoted to this problem.
 	
 	It is proved in \cite[(7.5.6)]{Grothendieck2}, that if $A$ and $B$ are complete noetherian 
 	local rings and their residue fields are finite
 	extensions of $k$ then, $A\hat{\otimes}_{k}B$ is a Cohen-Macaulay ring if $A$ and $B$ are Cohen-Macaulay and it is regular if $A$ and $B$ are regular and $k$ is perfect. In \cite{Watanabe}, Watanabe, Ishikawa, Tachibana and Otsuka showed that if $A$ and $B$
 	are Gorenstein semi-local rings such that $B/\mathfrak{n}$ is a finite
 	extension of $k$ for every maximal ideal $\mathfrak{n}$ of $B$, then $A\hat{%
 		\otimes}_{k}B$ is a Gorenstein semi-local ring. Recently, Shaul generalized
 	this last result in \cite{Shaul}, where he shows that, if $A$ and $B$ are preadic Gorenstein
 	rings of finite dimension and if the ring $A\hat{\otimes}_{k}B$ is
 	noetherian of finite dimension then $A\hat{\otimes}_{k}B$ is a Gorenstein
 	ring. In Section 4 we show that, if $A$ and $B$ are preadic noetherian and if $A\hat{\otimes}_{k}B$ is noetherian, then $A\hat{\otimes}%
 	_{k}B$\ is complete intersection (resp. Gorenstein, resp. Cohen-Macaulay) if
 	and only if, for every open maximal ideal $\mathfrak{m}$\ of $A$ and every open
 	maximal ideal $\mathfrak{n}$\ of $B$ the rings $A_{\mathfrak{m}}$ and $B_{%
 		\mathfrak{n}}$ are complete intersection (resp. Gorenstein, resp.
 	Cohen-Macaulay) and it is almost Cohen-Macaulay if and only if, $A_{%
 		\mathfrak{m}}$ is Cohen-Macaulay for every open maximal ideal $\mathfrak{m}$ of $A$
 	and $B_{\mathfrak{n}}$ is almost Cohen-Macaulay for every open maximal ideal $%
 	\mathfrak{n}$\ of $B$, or, $B_{\mathfrak{n}}$ is Cohen-Macaulay for every
 	open maximal ideal $\mathfrak{n}$ of $B$ and $A_{\mathfrak{m}}$ is almost
 	Cohen-Macaulay for every open maximal ideal $\mathfrak{m}$ of $A$, and if $k$ is perfect it is regular if and only if, for every open maximal ideal $\mathfrak{m}$ of $A$ and every open maximal ideal $\mathfrak{n}$ of $B$ the rings $A_{\mathfrak{m}}$ and $B_{\mathfrak{n}}$ are regular. If $A$ and $B$ are discrete, we also find the characterizations given by Tousi and Yassemi \cite{Tousi}.
 	
 	On the other hand, it is proved in \cite[(7.5.7)]{Grothendieck2} that if $k$ is perfect (resp.
 	algebraically closed), $A$ and $B$ are complete noetherian local and their
 	residue fields are finite extensions of $k$, then if $A$ and $B$ are reduced
 	(resp. domains) so is $A\hat{\otimes}_{k}B$. In Section 5 we show that if 
 	$k$ is perfect, $A$ and $B$ are noetherian semi-local and $A\hat{\otimes}%
 	_{k}B$ is noetherian, then $A\hat{\otimes}_{k}B$ is reduced (resp. normal)
 	if and only if $A$ and $B$ are analytically reduced (resp. analytically
 	normal) (for a generalization cf. Remarks 5.9 (i)  and \cite[Remarque (7.5.8)]{Grothendieck2}), and if $k$ is
 	algebraically closed, $A$ and $B$ are noetherian local and $A\hat{\otimes}%
 	_{k}B$ is noetherian, then $A\hat{\otimes}_{k}B$ is a domain if and only if $%
 	A$ and $B$ are analytically irreducible.
 	
 	Throughout we use the results from \cite{Matsumura}, \cite{Avramov} and \cite[\textbf{0}, (7.7)]{GrothendieckDieudonne}.
 	
 	\section{Preliminaries}
 	Let $k$ be a ring, $A$ and $B$ two preadic $k$-algebras with
 	ideals of definition $\mathfrak{a}$ and $\mathfrak{b}$. The
 	tensor product topology on $A\otimes _{k}B$ is defined by the ideals $%
 	\mathfrak{a}^{i}(A\otimes _{k}B)+\mathfrak{b}^{j}(A\otimes _{k}B)$. It is
 	preadic with ideal of definition $%
 	\mathfrak{a}(A\otimes _{k}B)+\mathfrak{b}(A\otimes _{k}B)$. The completed tensor product $A$ $\hat{\otimes}_{k}B$ of $A$
 	and $B$ is the completion of $A\otimes _{k}B$\ for this topology
 	and it is isomorphic to $\widehat{A}\hat{\otimes}_{k}\widehat{B}$ .
 	
 	\begin{lemma}[{cf. \cite[V.B.2]{Serre}}]    Let $k$ be a ring, $A$ and $B$ two preadic $%
 		k $-algebras with ideals of definition $\mathfrak{a}$ and $\mathfrak{b}$. Let $\mathfrak{r}$ be the ideal $%
 		\mathfrak{a}(A\otimes _{k}B)+\mathfrak{b}(A\otimes _{k}B)$
 		of $A\otimes _{k}B$ and $E$ the ring $A$ $\hat{\otimes}_{k}B$.    
 		If  $\mathfrak{a}$ and  $\mathfrak{b}$ are finitely generated, then 
 		{\footnotesize \ }
 		\begin{itemize}
 		\item[(i)] $E$ is adic with ideal of definition $\mathfrak{r}E$.
 		
 		\item[(ii)] $\mathfrak{r}E$ is contained in the radical of $E$ and the maximal ideals of $E$ correspond to those of $(A/a)\otimes_k(B/b)$.
 		
 		\item[(iii)] $E$ is noetherian if and only if $(A/\mathfrak{a})\otimes _{k}(B/%
 		\mathfrak{b})${\footnotesize \ } is noetherian.
 		\end{itemize}
 	\end{lemma}

 	It follows from (ii) that if $\mathfrak{M}$ is a maximal ideal of $E$, its
 	inverse image in $A$ contains $\mathfrak{a}$ and its inverse image in $B$
 	contains $\mathfrak{b}$.
 	
 	Suppose that the $k$-algebra $A/\mathfrak{a}$ is finite then, by (iii) ( resp. (ii)), if $B/\mathfrak{b}$ is noetherian (resp. semi-local) then $E$ is noetherian (resp. semi-local ).

 	\begin{lemma}    Let $k$ be a ring, $A$ and $B$ two preadic $%
 		k $-algebras with finitely generated ideals of definition, $C$ the ring $A\otimes _{k}B$ and $E$ the
 		ring $A\hat{\otimes}_{k}B$. Suppose that $E$ is noetherian. Then                                                                 
 		\begin{itemize}
 			\item[(i)] If $\mathfrak{r}^{\prime}$ is an ideal of $C$, then   $E/\mathfrak{r}^{\prime} E$ is the 
 			completion of $C/\mathfrak{r}^{\prime}$.
 			\item[(ii)] If  $\mathfrak{a}^{\prime }$ and $\mathfrak{b}^{\prime }$ are ideals of $A$ and $B$, then
 			\begin{itemize}
 			 \item[(a)]  We have $$E/(\mathfrak{a}^{\prime }E+\mathfrak{b}^{\prime }E)\cong(A/\mathfrak{a}^{\prime })\hat{\otimes}_{k}(B/\mathfrak{b}^{\prime }).$$
 			 
 			\item[(b)] If $k$ is a field, $\mathfrak{a}^{\prime }$ is open in $A$ and $\mathfrak{b}^{\prime }$
 			is open in $B$, then $\mathfrak{a}^{\prime }E+\mathfrak{b}^{\prime }E=E$ if
 			and only if $\mathfrak{a}^{\prime }=A$ or $\mathfrak{b}^{\prime }=B$.
 			\end{itemize}
 		\end{itemize}
 	\end{lemma}
 	
 	\begin{proof}
 		(i)  If $i:C\to E$ is the canonical homomorphism, it suffices to show that $\mathfrak{r}^{\prime}E=\overline{i(\mathfrak{r}^{\prime})}$. It is clear that $\mathfrak{r}^{\prime}E\subset\overline{i(\mathfrak{r}^{\prime})}$. On the other hand, since $E$ is noetherian, by Lemma 2.1 (ii), it is a Zariski ring. Hence the ideal $\mathfrak{r}^{\prime}E$ of $E$ is closed in $E$, and therefore $\overline{i(\mathfrak{r}^{\prime})}\subset \mathfrak{r}^{\prime}E$.
 		
 		(ii) (a) The canonical isomorphism  $C/(\mathfrak{a}^{\prime }C+\mathfrak{b}^{\prime }C)\rightarrow(A/\mathfrak{a}^{\prime })\otimes_k(B/\mathfrak{b}^{\prime })$ is a topological isomorphism. By (i),  
 		$E/(\mathfrak{a}^{\prime }E+\mathfrak{b}^{\prime }E)$ is the completion of $C/(\mathfrak{a}^{\prime }C+\mathfrak{b}^{\prime }C)$. Hence the induced homomorphism $E/(\mathfrak{a}^{\prime }E+\mathfrak{b}^{\prime }E)\rightarrow(A/\mathfrak{a}^{\prime })\hat\otimes_k(B/\mathfrak{b}^{\prime })$ is an isomorphism. 
 		
 		(b) The rings $A/\mathfrak{a}^{\prime }$ and $B/\mathfrak{b}^{\prime }$ are discrete, hence by (a) $E/(\mathfrak{a}^{\prime }E+\mathfrak{b}^{\prime }E)$ is isomorphic to $(A/\mathfrak{a}^{\prime })\otimes_k(B/\mathfrak{b}^{\prime })$. The equivalence follows from this isomorphism.	
 	\end{proof}

 	We find \cite[Proposition 2.2.8]{Tavanfar}.

 	\section{The key result} 
 	
 	\begin{theorem}  Let $k$ be a field, $A$ and $B$ two preadic
 		noetherian $k$-algebras with ideals of definition  $\mathfrak{a}$ and  $%
 		\mathfrak{b}$. Let $E$\ be the ring$\ A\hat{\otimes}%
 		_{k}B$. If $E$ is noetherian, then
 		
 		(i) The homomorphism $A\rightarrow E${\footnotesize \ }\ is flat.
 		
 		(ii) The homomorphism $A\rightarrow E$ is faithfully flat if and only if $A$%
 		\ is a Zariski ring.
 		
 		(iii) If $\mathfrak{a}^{\prime }$ is an ideal of $A$ containing $\mathfrak{a}
 		$ then l'homomorphism $B\rightarrow E/\mathfrak{a}^{\prime }E$ is flat.%
 		 	\end{theorem}
 	
 	\begin{proof} (i) This follows from ~\cite[Corollary 5.4]{Tabaa2}.
 		
 		(ii) Let $\mathfrak{m}$ be a maximal ideal of $A$. If the homomorphism $%
 		A\rightarrow E$ is faithfully flat there exists a maximal ideal $\mathfrak{M}
 		$ of $E$ such that $\mathfrak{m}$ is its inverse image in $A$; $\mathfrak{M}$
 		contains $\mathfrak{a}E$ and therefore $\mathfrak{m}$ contains $\mathfrak{a}$%
 		. Conversely, if $A$ is a Zariski ring, $\mathfrak{m}$ contains $\mathfrak{a}
 		$, so, by Lemma 2.2 (ii) (b) $\mathfrak{m}E+\mathfrak{b}E\neq E$ and a fortiori $%
 		\mathfrak{m}E\neq E$.
 		
 		(iii)\ Let $C$ denote the ring $A\otimes _{k}B$. By Lemma 2.2 (i), $E/%
 		\mathfrak{a}^{\prime }E$ is the completion of $C/\mathfrak{a}%
 		^{\prime }C$. Since $\mathfrak{a}^{\prime }$\ contains $\mathfrak{a}$, the
 		topology of $C/\mathfrak{a}^{\prime }C$  is the $%
 		\mathfrak{b}$-preadic topology, hence $E/\mathfrak{a}^{\prime }E$ is the
 		completion of $C/\mathfrak{a}^{\prime }C$ for this topology. But, $%
 		C/\mathfrak{a}^{\prime }C$ is isomorphic to $(A/\mathfrak{a}^{\prime
 		})\otimes _{k}B$ and $A/\mathfrak{a}^{\prime }$ is flat over $k$, so, $C/%
 		\mathfrak{a}^{\prime }C$\ is flat over $B$. By \cite[Theorem 3.2]{Tabaa2}, $E/%
 		\mathfrak{a}^{\prime }E$ is also flat over $B$.
 	\end{proof}
 	
 	\begin{corollary}    Let $k$ be a field, $A$ and $B$ two
 		noetherian local $k$-algebras with maximal ideals  $\mathfrak{m}$ and $%
 		\mathfrak{n }$. Let $E$\ be the ring $\ A\hat{\otimes}_{k}B$.
 		If $E$ is noetherian, then
 		\begin{itemize}

 		\item[(i)] $\dim(E)=\dim (A)+\dim (B)+\inf (\ tr.\deg _{k}(A/\mathfrak{m}),tr.\deg
 		_{k}(B/\mathfrak{n})).$
 		
 		\item[(ii)] $\dim (E)$ is finite if and only if, $tr.\deg _{k}(A/%
 		\mathfrak{m})$ or $tr.\deg _{k}(B/\mathfrak{n}_{s})$ is finite.
 		\end{itemize}
 	\end{corollary}
 	
 	\begin{proof}  (i) Let $\mathfrak{Q}$ be a maximal ideal of $E$. The inverse image
 		of $\mathfrak{Q}$ in $A$ is $\mathfrak{m}$ and the inverse image of $%
 		\mathfrak{Q}$ in $B$ is $\mathfrak{n}$. By the above theorem, the local
 		homomorphisms $A\rightarrow E_{\mathfrak{Q}}$ and $B\rightarrow E_{\mathfrak{%
 				Q}}/\mathfrak{m}E_{\mathfrak{Q}}$ are flat. The closed fiber of the second
 		homomorphism is isomorphic to $(E/\mathfrak{r}E)_{\mathfrak{Q}/\mathfrak{r}%
 			E} $, where $\mathfrak{r}$ is the ideal $(\mathfrak{a}\otimes
 		_{k}B)+(A\otimes _{k}\mathfrak{b})$ of $A\otimes _{k}B.$ Hence
 		$$
 		\dim(E_{\mathfrak{Q}})=\dim(A)+\dim(E_{\mathfrak{Q}}/\mathfrak{m}E_{\mathfrak{Q}})
 		$$
 		and
 		$$
 		\dim(E_{\mathfrak{Q}}/\mathfrak{m}E_{\mathfrak{Q}})=\dim(B)+\dim ((E/\mathfrak{r}E)_{\mathfrak{Q}/\mathfrak{r}E}),
 		$$
 		whence
 		$$
 		\dim(E_{\mathfrak{Q}})=\dim(A)+\dim(B)+\dim((E/\mathfrak{r}E)_{\mathfrak{Q}/\mathfrak{r}E}).
 		$$ 
 		We deduce that 
 		$$
 		\dim(E)=\dim(A)+\dim(B)+\dim(E/\mathfrak{r}E).
 		$$ 
 		On the other hand, by \cite{Sharp1} we have
 		$$
 		\dim((A/\mathfrak{m})\otimes_{k}(B/\mathfrak{n}))= \inf (\ tr.\deg _{k}(A/\mathfrak{m}),tr.\deg
 		_{k}(B/\mathfrak{n})),
 		$$ 
 		and since $E/\mathfrak{r}E$ is isomorphic to $(A/\mathfrak{m})\otimes_{k}(B/\mathfrak{n})$, then
 		$$
 		\dim(E/\mathfrak{r}E)= \inf (\ tr.\deg _{k}(A/\mathfrak{m}),tr.\deg
 		_{k}(B/\mathfrak{n})).
 		$$
 		Whence
 		$$
 		\dim(E)=\dim(A)+\dim(B)+ \inf (\ tr.\deg _{k}(A/\mathfrak{m}),tr.\deg
 		_{k}(B/\mathfrak{n})).
 		$$ 
 		
 		(ii) This follows from (i).
 	\end{proof}
 	
 	\begin{corollary}    Let $k$ be a field, $A$ and $B$ two preadic
 		noetherian $k$-algebras, $E$\ the ring $\ A\hat{\otimes}_{k}B$ and let $\mathbf{R}$ be a
 		local property of noetherian local rings which descends by faithful
 		flatness. Consider the following conditions:
 		\begin{itemize}
 		\item[(i)] $E$ satisfies $\mathbf{R}$.
 		
 		\item[(ii)] $\widehat{A}$ satisfies $\mathbf{R}$.
 		
 		\item[(iii)] For every open maximal ideal $\mathfrak{m}$ of $A$ the ring $A_{%
 			\mathfrak{m}}$ satisfies $\mathbf{R}$.
 		\end{itemize}
 		\hspace{.5cm}	
 		If $E$ is noetherian, then (i) $\Rightarrow $ (ii) $\Rightarrow $ (iii).
 	\end{corollary}
 	
 	\begin{proof}	  Since $\widehat{A}$ is a Zariski ring, the implication (i) $ %
 		\Rightarrow $ (ii) follows from the fact that, by Theorem 3.1 (ii), the
 		homomorphism $\widehat{A}\rightarrow E$ is faithfully flat.
 		
 		The implication (ii) $\Rightarrow $ (iii) follows from \cite[Proposition 6.3]{Greco}.
 	\end{proof}

 	\section{Complete intersection, Gorenstein, Cohen-Macaulay and almost Cohen-Macaulay Rings.}
 	
 	\begin{lemma}     Let $k$ be a field, $L$ and $M$ two extension fields
 		of $k$. Suppose that  $L\otimes _{k}M$ is noetherian. Then
 		
 		(i) $L\otimes _{k}M$ is complete intersection.
 		
 		(ii) If $k$ is perfect, then $L\otimes _{k}M$ is regular.
 	\end{lemma}
 	
 	\begin{proof}	 
 		Proof. Cf. \cite[Proposition 1.5 (a)]{Tousi} for (i) and \cite[Note]{Sharp2} for (ii).
 	\end{proof}
 	
 	\begin{lemma}      Let $k$ be a field, $A$ and $B$ two preadic noetherian $%
 		k $-algebras  and $E$\ the ring $%
 		\ A\hat{\otimes}_{k}B$. Let $\mathfrak{M}$ be a maximal ideal of $E$, $\mathfrak{p}$ and  $\mathfrak{q}$ its inverse images in $A$ and $B$. Suppose that $E$ is noetherian. Then

 		(i) The ring $k(\mathfrak{p})\otimes _{k}k(\mathfrak{q})$ is noetherian.
 		
 		(ii) If \,  $\mathfrak{M}^{\prime }$ is the inverse image of \, $\mathfrak{M}$ in $A\otimes _{k}B$,   then

 		\begin{equation*}
 		E_{%
 			\mathfrak{M}}/(\mathfrak{p}E_{\mathfrak{M}}+\mathfrak{q}E_{\mathfrak{M}}) \cong (\kappa(\mathfrak{p})\otimes _{k}\kappa(\mathfrak{q}))_{_{\mathfrak{M%
 				}^{\prime }(\kappa(\mathfrak{p})\otimes _{k}\kappa(\mathfrak{q}))}} . 
 		\end{equation*}
 	\end{lemma}
 	
 	\begin{proof}  Let $\mathfrak{a}$ and $\mathfrak{b}$ be ideals of definition of $A$ and $B$.\
 		
 		(i) By Lemma 2.1 (iii), the ring $(A/\mathfrak{a})\otimes _{k}(B/%
 		\mathfrak{b})${\footnotesize \ }is noetherian. Since $\mathfrak{p}$ contains $\mathfrak{a}$ and $\mathfrak{q%
 		}$ contains $\mathfrak{b}$, the ring $(A/\mathfrak{p}%
 		)\otimes _{k}(B/\mathfrak{q})${\footnotesize \ }is noetherian. As $\kappa(%
 		\mathfrak{p})\otimes _{k}\kappa(\mathfrak{q})$ is a ring of fractions of $(A/%
 		\mathfrak{p})\otimes _{k}(B/\mathfrak{q})$ \ then it is also noetherian.
 		
 		(ii) Let $C$ denote the ring $A\otimes _{k}B$. It is clear that $\mathfrak{M}$ contains $\mathfrak{p}E+\mathfrak{q}E$ and $%
 		\mathfrak{M}^{\prime }$ contains $\mathfrak{p}C+\mathfrak{q}C$. Since  $C_{\mathfrak{M^{\prime }}}\cong(A_{\mathfrak{p}}\otimes _{k}B_{\mathfrak{q}})_{\mathfrak{M}^{\prime }(A_{\mathfrak{p}}\otimes _{k}B_{\mathfrak{q}})}$ then $(C /(\mathfrak{p}C+\mathfrak{q}C))_{%
 			\mathfrak{%
 				M}^{\prime }/(\mathfrak{p}C+\mathfrak{q}C)}\cong (\kappa(\mathfrak{p})\otimes_{k}\kappa(\mathfrak{q}))_{_{\mathfrak{M}^{\prime }(\kappa(\mathfrak{p})\otimes _{k}\kappa(\mathfrak{q}))}}$.    
 		On the other hand, since $\mathfrak{p}$ contains $\mathfrak{a}$ and $\mathfrak{q%
 		}$ contains $\mathfrak{b}$, the topology of $C/(\mathfrak{p}C+\mathfrak{q}C)$ is the discrete topology. By Lemma 2.2 (i), the canonical homomorphism $C/(\mathfrak{p}C+\mathfrak{q}C)\rightarrow E/(\mathfrak{p}E+\mathfrak{q}E)$ is an isomorphism. As $\mathfrak{%
 			M}^{\prime }/(\mathfrak{p}C+\mathfrak{q}C)$ is the inverse image of $\mathfrak{%
 			M}/(\mathfrak{p}C+\mathfrak{q}C)$ in $C/(\mathfrak{p}C+\mathfrak{q}C)$, hence the induced homomorphism \ $(C /(\mathfrak{p}C+\mathfrak{q}C))_{%
 			\mathfrak{%
 				M}^{\prime }/(\mathfrak{p}C+\mathfrak{q}C)}\rightarrow (E /(\mathfrak{p}E+\mathfrak{q}E))_{%
 			\mathfrak{%
 				M}/(\mathfrak{p}E+\mathfrak{q}E)}$ is also an isomorphism. 
 	\end{proof}

 	As in \cite{Kang}, a noetherian ring $A$ is called almost Cohen-Macaulay if 
 	 $\mbox{depth}(\mathfrak{p},A)=\mbox{depth}(\mathfrak{p}A_{\mathfrak{p}},A_{\mathfrak{p}})$
 	for every prime ideal $\mathfrak{p}$ of $A$. By \cite[(2.6)]{Kang}, $A$ is almost
 	Cohen-Macaulay if and only if $A_{\mathfrak{p}}$ (resp. $A_{\mathfrak{m}}$)
 	is almost Cohen-Macaulay for every prime ideal $\mathfrak{p}$ (resp.
 	maximal ideal $\mathfrak{m}$) of $A$.
 	
 	\begin{theorem}     Let $k$ be a field, $A$ and 
 		$B$ two preadic noetherian $k$-algebras and $E$\ the ring$\ A\hat{\otimes}%
 		_{k}B$. Suppose that $E$ is noetherian. Then\textbf{\ }
 		\begin{itemize}

 		\item[(i)] $E$ is complete intersection (resp. Gorenstein, resp. Cohen-Macaulay) if
 		and only if, for every open maximal ideal $\mathfrak{m}$\ of A and every open maximal ideal $%
 		\mathfrak{n}$\ of $B$ the rings $A_{\mathfrak{m}}$ and $B_{\mathfrak{n}}$
 		are complete intersection (resp. Gorenstein, resp. Cohen-Macaulay).
 		
 		\item[(ii)] If $E$ is regular then for every open maximal ideal $\mathfrak{m}$\ of $A$ and
 		every open maximal ideal $\mathfrak{n}$\ of $B$ the rings $A_{\mathfrak{m}}$ and $B_{%
 			\mathfrak{n}}$ are regular. This condition is sufficient if $k$ is perfect.
 		
 		\item[(iii)] $E$ is almost Cohen-Macaulay if and only if, $A_{\mathfrak{m}}$ is
 		Cohen-Macaulay for every open maximal ideal $\mathfrak{m}$ of A and $B_{\mathfrak{n}%
 		} $ is almost Cohen-Macaulay for every open maximal ideal $\mathfrak{n}$\ of $B$,
 		or, $B_{\mathfrak{n}}$ is Cohen-Macaulay for every open maximal ideal $\mathfrak{n}$
 		of B and $A_{\mathfrak{m}}$ is almost Cohen-Macaulay for every open maximal ideal $%
 		\mathfrak{m}$ of $A$.
 		\end{itemize}
 	\end{theorem}
 	
 	\begin{proof}	 
 		Let $\mathfrak{a}$ and $\mathfrak{b}$\ be  ideals of definition of $A$
 		and $B$, respectively.
 		
 		We prove that the conditions are sufficient. It is clear that we can
 		replace in the hypotheses \textquotedblleft maximal\textquotedblright\ by
 		\textquotedblleft prime\textquotedblright . Let $\mathfrak{M}$ be a maximal
 		ideal of $E$, $\mathfrak{p}$ its inverse image in $A$ and $\mathfrak{q}$ its
 		inverse image in $B$; $\mathfrak{p}$ contains $\mathfrak{a}$ and $\mathfrak{q%
 		}$ contains $\mathfrak{b}$. By Theorem 3.1, the local homomorphisms\ $A_{%
 			\mathfrak{p}}\rightarrow E_{\mathfrak{M}}$ and $B_{\mathfrak{q}}\rightarrow
 		E_{\mathfrak{M}}/\mathfrak{p}E_{\mathfrak{M}}$ are flat.The closed fiber of 
 		the homomorphism $B_{\mathfrak{q}}\rightarrow E_{%
 			\mathfrak{M}}/\mathfrak{p}E_{\mathfrak{M}}$ is isomorphic to $E_{\mathfrak{M}%
 		}/(\mathfrak{p}E_{\mathfrak{M}}+\mathfrak{q}E_{\mathfrak{M}})$. By Lemma 4.2
 		(ii), this fiber is isomorphic to a localization of $\kappa(\mathfrak{p})\otimes _{k}\kappa(%
 		\mathfrak{q})$ and by Lemma 4.2 (i), the ring $\kappa(\mathfrak{p})\otimes _{k}\kappa(%
 		\mathfrak{q})$ is noetherian. Hence, it follows from Lemma 4.1 that  this fiber is complete
 		intersection and if $k$ is perfect it is regular.\qquad\ 
 		
 		(i) and (ii) follow by using the homomorphism $B_{\mathfrak{q}}\rightarrow
 		E_{\mathfrak{M}}/\mathfrak{p}E_{\mathfrak{M}}$ and then the homomorphism $A_{%
 			\mathfrak{p}}\rightarrow E_{\mathfrak{M}}$.
 		
 		(iii) follows by applying \cite[Lemme 3.1]{Tabaa1} to the homomorphism $B_{\mathfrak{q%
 		}}\rightarrow E_{\mathfrak{M}}/\mathfrak{p}E_{\mathfrak{M}}$ and then to the
 		homomorphism $A_{\mathfrak{p}}\rightarrow E_{\mathfrak{M}}$.
 		
 		We prove that the conditions are necessary.
 		
 		(i) and (ii) follow from Corollary 3.3.
 		
 		(iii) Let $\mathfrak{m}$ and $\mathfrak{n}$\ be open maximal ideals of $%
 		A $ and $B$; $\mathfrak{m}$ contains $\mathfrak{a}$ and $\mathfrak{n}$
 		contains $\mathfrak{b}$. By Lemma 2.2 (ii) (b), there exists a maximal ideal $%
 		\mathfrak{Q}$ of $E$ such that $\mathfrak{m}$ is the inverse image of $%
 		\mathfrak{Q}$ in $A$ and $\mathfrak{n}$ is the inverse image of $\mathfrak{Q}
 		$ in $B$. By \cite[Lemme 3.1]{Tabaa1}, applied to the homomorphism $A_{\mathfrak{m}%
 		}\rightarrow E_{\mathfrak{Q}}$ and then to the homomorphism $B_{\mathfrak{n}%
 		}\rightarrow E_{\mathfrak{Q}}/\mathfrak{m}E_{\mathfrak{Q}}$, one of the
 		rings, $A_{\mathfrak{m}}$ or $B_{\mathfrak{n}}$, is Cohen-Macaulay and the
 		other is almost Cohen-Macaulay.	
 		\end{proof}
 	
 	\begin{lemma}  
 		Let $k$ be a field and $A$ a complete
 		noetherian local $k$-algebra. Then the structural homomorphism $k\rightarrow
 		A$ has a factorization $k\rightarrow C\stackrel{u}{\rightarrow} A$, where $C$ is a regular
 		local ring and the homomorphism $u$ is local surjective.
 	\end{lemma}
 	
 	\begin{proof}  Let \textbf{P} be the prime subfield of $k$ and $K$ the
 		residue field of $A$. By Cohen Theorem \cite[(19.8.8)]{Grothendieck1}, there exists a
 		complete regular local ring $C$ containing $K$ such that $A$ is isomorphic
 		to $C/\mathfrak{J}$, where $\mathfrak{J}$ is an ideal of $C$. The \textbf{P}%
 		-algebra $C$ is separated and complete for the $\mathfrak{J}$-preadic
 		topology and, by \cite[(19.6.1)]{Grothendieck1} the extension $k$ of \textbf{P} is \textbf{P}%
 		-formally smooth. Hence, it follows from \cite[(19.3.10)]{Grothendieck1} that the \textbf{P}-homomorphism $%
 		k\rightarrow C/\mathfrak{J}$ has a factorization $k \stackrel{\sigma}{\rightarrow }
 		C \rightarrow C/\mathfrak{J}$, where $\sigma $ is a \textbf{P}%
 		-homomorphism.
 	\end{proof}
 	
 	The following result generalizes \cite[Corollary 2.12]{Shaul}.
 	
 	\begin{corollary}    Let $k$ be a field, $A$ and $B$ two
 		noetherian local $k$-algebras. If the ring $A\hat{\otimes}_{k}B$ is
 		noetherian (resp. noetherian of finite dimension), then it is a quotient of
 		a complete intersection (resp. complete intersection of finite dimension)
 		ring and if moreover $k$ is perfect then it is a quotient of a regular
 		(resp. regular of finite dimension) ring.
 		
 	\end{corollary}
 	
 	\begin{proof}  We may assume that $A$ and $B$ are complete. By the previous lemma,
 		the structural homomorphism $k\rightarrow A$ (resp. $k\rightarrow B$ )
 		has a factorization $k\rightarrow C \stackrel{u}{\rightarrow}A$ (resp. $k\rightarrow
 		D\stackrel{v}{\rightarrow}B$\ ), where $C$ (resp. $D$) is a regular local ring and
 		the homomorphism $u$ (resp. $v$) is local surjective. The homomorphism $%
 		u\otimes v:C\otimes _{k}D\rightarrow A\otimes _{k}B$ is surjective and, by
 		the definition of the tensor product topology, it is strict. So, the
 		homomorphism $u\hat{\otimes}v:C\hat{\otimes}_{k}D\rightarrow A\hat{\otimes}%
 		_{k}B$ induced by $u\otimes v$ is surjective. On the other hand, the ring $A%
 		\hat{\otimes}_{k}B$ is noetherian, so, by Lemma 2.1 (iii), the ring $C\hat{%
 			\otimes}_{k}D$ is also noetherian and since $C$ and $D$ are regular, it
 		follows from Theorem 4.3 that $C\hat{\otimes}_{k}D$ is complete intersection
 		and that if $k$ is perfect it is regular. If $A\hat{\otimes}_{k}B$ is of
 		finite dimension it follows from Corollary 3.2 (ii) that $C\hat{\otimes}%
 		_{k}D $ is also of finite dimension.
 	\end{proof} 
 	
 	We deduce that if $A\hat{\otimes}_{k}B$ is noetherian of finite dimension
 	then it possesses a dualizing complex.
 	
 	\section{ Reduced, Normal Rings and Domains }
 	
 	Recall that a homomorphism $\sigma :A\rightarrow S$ of noetherian rings is
 	reduced (resp. normal) if it is flat and, if for every prime ideal $%
 	\mathfrak{p}$ of $A$ the fiber $S\otimes _{A}k(\mathfrak{p})$ of $\sigma $\
 	is geometrically reduced (resp. geometrically normal) over $k(\mathfrak{p})$.
 	
 	Unlike the previous cases, the reduced case and the normal case essentially use
 	the following localization theorem.
 	
 	\textbf{\ }
 	
 	\begin{theorem}[\cite{Nishimura}]  Let $\sigma :A\rightarrow S$ be a flat
 		local homomorphism of noetherian local rings and let $\mathfrak{m}$ be the maximal
 		ideal of $A$. Suppose that the formal fibers of $A$ are geometrically reduced
 		(resp. geometrically normal). If the closed fiber $S/\mathfrak{m}S$ is
 		geometrically reduced (resp. geometrically normal) over $A/\mathfrak{m}$,
 		then $\sigma $\ is reduced (resp. normal).
 	\end{theorem}
 	
 	We deduce the following corollary.
 	
 	\textbf{\ }
 	
 	\begin{corollary}   Let $\sigma :A\rightarrow S$ be a flat
 		homomorphism of noetherian rings. Assume that $A$ is complete local with
 		maximal ideal $\mathfrak{m}$ and that $\mathfrak{m}S$ is contained in the
 		radical of $S$. If the closed fiber $S/\mathfrak{m}S$ is geometrically
 		reduced (resp. geometrically normal) over $A/\mathfrak{m}$, then $S$ is
 		reduced (resp. normal) if $A$ is reduced (resp. normal).
 	\end{corollary}
 	
 	\begin{proof}  Let $\mathfrak{Q}$ be a maximal ideal of $S$. We have $\sigma ^{-1}(%
 		\mathfrak{Q})=\mathfrak{m}$, so the homomorphism $\widetilde{\sigma }%
 		:A\rightarrow S_{\mathfrak{Q}}$ induced by $\sigma $\ is local and flat. Its
 		closed fiber is $(S/\mathfrak{m}S)_{\mathfrak{Q}/\mathfrak{m}S}$. It is
 		geometrically reduced (resp. geometrically normal) over $A/\mathfrak{m}$. By
 		the above theorem, the homomorphism $\widetilde{\sigma }$ is reduced (resp.
 		normal) and since $A$ is reduced (resp. normal) $S_{\mathfrak{Q}}$ is also
 		reduced (resp. normal).
 	\end{proof}

 	\begin{lemma} 
 		Let $k$ be a ring, $A$ and $B$ two semi-local $k$-algebras, $(\mathfrak{m}_{r})_{r}$ and $(\mathfrak{n}_{s})_{s}$  the distinct maximal ideals of $A$ and $B$. Then 
 		\begin{equation*}
 		A\hat{\otimes}_{k}B \cong \prod\limits_{r,s}(\widehat{A}_{\mathfrak{m}_{r}} \hat{\otimes}_{k}\widehat{B}_{\mathfrak{n}_{s}}). 
 		\end{equation*}
 	\end{lemma}

 	\begin{proof} 
 		It is similar to that of \cite[Theorem 8.15]{Matsumura}.
 	\end{proof}
 	
 	In all the rest $E$ denotes the ring $A\hat{\otimes}_{k}B$.
 	
 	\begin{theorem}      Let $k$ be a perfect field, $A$ and $B$ two
 		noetherian semi-local $k$-algebras such that the ring $A\hat{\otimes}_{k}B$
 		is noetherian. Then $A\hat{\otimes}_{k}B$ is reduced (resp. normal) if and
 		only if $A$ and $B$ are analytically reduced (resp. normal).
 	\end{theorem} 
 	
 	\begin{proof}  The necessary condition follows from Corollary 3.3. We prove that the                                              condition is sufficient. We may suppose that $A$ and $B$ are complete.\
 		
 		The case where $A$ and $B$ are local. Let $\mathfrak{m}$ and $\mathfrak{n}$
 		be the maximal ideals of $A$ and $B$.\ By Theorem 3.1 (i), the 
 		homomorphism $A\rightarrow \ E$ is flat. Since $A$ is reduced (resp.
 		normal) complete and $\mathfrak{m}E$ is contained in the radical of $E$, to
 		show that $E$ is reduced (resp. normal) it suffices, by the previous corollary,
 		to show that the closed fiber $E/\mathfrak{m}E$ is geometrically reduced
 		(resp. geometrically normal) over $A/\mathfrak{m}$. Let $L$ be a finite
 		extension of $A/\mathfrak{m}$. We prove that the ring $L\otimes _{A/%
 			\mathfrak{m}}(E/\mathfrak{m}E)$ is reduced (resp. normal). For this,
 		consider the homomorphism $B\rightarrow L\otimes _{A/\mathfrak{m}}(E/%
 		\mathfrak{m}E)$. By Theorem 3.1 (iii), the homomorphism $B\rightarrow E/%
 		\mathfrak{m}E$ is flat. Since the homomorphism $A/\mathfrak{m}\rightarrow
 		L$ is flat, the homomorphism $(E/\mathfrak{m}E)\rightarrow L\otimes _{A/%
 			\mathfrak{m}}(E/\mathfrak{m}E)$ is also flat. We deduce that the
 		homomorphism $B\rightarrow L\otimes _{A/\mathfrak{m}}(E/\mathfrak{m}E)$ is
 		flat. We show that $\mathfrak{n}(L\otimes _{A/\mathfrak{m}}(E/\mathfrak{m}%
 		E)) $ is contained in the radical of $L\otimes _{A/\mathfrak{m}}(E/\mathfrak{%
 			m}E) $. Let $\mathfrak{Q}$ be a maximal ideal of $L\otimes _{A/\mathfrak{m}%
 		}(E/\mathfrak{m}E)$. Since the homomorphism $A/\mathfrak{m}\rightarrow L$ is
 		integral, the homomorphism $E\rightarrow L\otimes _{A/\mathfrak{m}}(E/%
 		\mathfrak{m}E)$ is also integral. So the inverse image $\mathfrak{Q}^{\prime
 		}$ of $\mathfrak{Q}$ in $E$\ is a maximal ideal of $E$. Hence $\mathfrak{n}E$
 		is contained in $\mathfrak{Q}^{\prime }$ and therefore $\mathfrak{n}%
 		(L\otimes _{A/\mathfrak{m}}(E/\mathfrak{m}E))$ is contained in $\mathfrak{Q}$%
 		. Let $F$ denote the closed fiber $(L\otimes _{A/\mathfrak{m}}(E/\mathfrak{m}%
 		E))\otimes _{B}(B/\mathfrak{n)}$ of the homomorphism $B\rightarrow L\otimes
 		_{A/\mathfrak{m}}(E/\mathfrak{m}E)$. We show that $F$ is geometrically
 		reduced (resp. geometrically normal) over $B/\mathfrak{n}$. Let $M$ be a
 		finite extension of $B/\mathfrak{n}$. Since $E/(\mathfrak{m}E+\mathfrak{n}E)$
 		is isomorphic to $(A/\mathfrak{m})\otimes _{k}(B/\mathfrak{n)}$, $F$ is
 		isomorphic to $L\otimes _{k}(B/\mathfrak{n)}$. Hence $F\otimes _{B/\mathfrak{%
 				n}}M$ is isomorphic to $L\otimes _{k}M$. Since $k$ is perfect, $L\otimes
 		_{k}M$ is regular and therefore  $F\otimes _{B/\mathfrak{n}}M$ is also
 		regular. So the fiber $F$ is geometrically regular over $B/\mathfrak{n}$. As 
 		$B$ is reduced (resp. normal) and complete, it follows from the previous
 		corollary that the ring $L\otimes _{A/\mathfrak{m}}(E/\mathfrak{m}E)$ is
 		reduced (resp. normal).\
 		
 		General case. Let $(\mathfrak{m}_{r})_{r}$ and $(\mathfrak{n}_{s})_{s}$ be the distinct maximal ideals of $A$ and $B$. In view of the previous lemma, it suffices to show that for all $r, s$, the ring $\widehat{A}_{\mathfrak{m}_{r}}$ $\hat{\otimes}_{k}\widehat{B}_{\mathfrak{n}_{s}}$ is reduced ( resp. normal ). We have $A\cong\prod\widehat{A}_{\mathfrak{m}_{r}}$ and $B\cong\prod\widehat{B}_{\mathfrak{n}_{s}}$. Since $A$ and $B$ are reduced ( resp. normal ), $\widehat{A}_{\mathfrak{m}_{r}}$ and $\widehat{B}_{\mathfrak{n}_{s}}$ are also reduced ( resp. normal ), and since $A\hat{\otimes}_{k}B$ is noetherian, it follows from the previous lemma that $\widehat{A}_{\mathfrak{m}_{r}}$ $\hat{\otimes}_{k}\widehat{B}_{\mathfrak{n}_{s}}$
 		is also noetherian. So, by the above case, $\widehat{A}_{\mathfrak{m}_{r}}$ $\hat{\otimes}_{k}\widehat{B}_{\mathfrak{n}_{s}}$ is reduced ( resp. normal ). 
 	\end{proof}

 	\begin{remark}  \upshape   The assertion (ii) of Theorem 4.3 and the above
 		theorem are not true if $k$ is not perfect. Let $p$ be the characteristic of 
 		$k$, $a$ be an element of $k-k^{p}$ and $K$ the extension $k[X]/(X^{p}-a)$ of 
 		$k$, where $X$ is an indeterminate over $k$. The noetherian ring $K\otimes
 		_{k}K$ is not reduced.
 	\end{remark}
 	
 	It remains the "domain" case. For this, we will need the two following lemmas.
 	
 	Recall that if $A$ is a ring, $Spec(A)$ is connected if and only if $A$
 	contains no idempotents than 0 and 1.
 	
 	\begin{lemma}    Let $k$ be a perfect field, $A$ and $B$ two complete
 		noetherian local $k$-algebras with maximal ideals  $\mathfrak{m}$ and $%
 		\mathfrak{n }$. Suppose that $A$ and $B$ are integrally closed and that $A\hat{\otimes}_{k}B$ is
 		noetherian. Then 
 		$A\hat{\otimes}_{k}B$ is a domain if and only if $Spec((A/\mathfrak{m}%
 		)\otimes _{k}(B/\mathfrak{n}))$ is connected.
 	\end{lemma}
 	
 	\begin{proof} The ring $E$ is locally a domain since, by the previous theorem, it is
 		normal. If $\mathfrak{r}$ is the ideal $(\mathfrak{m}\otimes
 		_{k}B)+(A\otimes _{k}\mathfrak{n})$ of $A\otimes _{k}B$, then by \cite[Proposition
 		10.8]{Greco}, the $\mathfrak{r}E$-adic ring $E$ is a domain if and only if $Spec(E/%
 		\mathfrak{r}E)$ is connected. The equivalence follows from the fact that $E/%
 		\mathfrak{r}E$ is isomorphic to $(A/\mathfrak{m})\otimes _{k}(B/\mathfrak{n}%
 		) $.
 	\end{proof}
 	
 	\begin{lemma}  Let $k$ be a ring, $A$ and $B$ two preadic noetherian
 		$k$-algebras such that the ring $A\hat{\otimes}_{k}B$ is noetherian
 		and let $\mathfrak{a}$ be an ideal of definition of $A$. For all finitely generated $%
 		A $-module $M$ with the $\mathfrak{a}$-preadic topology, the completed
 		tensor product $M\hat{\otimes}_{k}B$ is identified with $M\otimes _{A}(A\hat{%
 			\otimes}_{k}B)$.
 	\end{lemma}
 	
 	\begin{proof} By \cite[\textbf{0}, (7.7.8)]{GrothendieckDieudonne}, $M\hat{\otimes}_{k}B$\ is identified
 		with $M\hat{\otimes}_{A}E$ , and by \cite[\textbf{0}, (7.7.9)]{GrothendieckDieudonne} applied to the
 		topological adic noetherian $A$-algebra $E$, $M\hat{\otimes}_{A}E$ is
 		identified with $M\otimes _{A}E$.
 	\end{proof}
 	
 	\begin{theorem}    Let $k$ be an algebraically closed field, $A$
 		and $B$ two noetherian local $k$-algebras such that the ring $A\hat{\otimes}%
 		_{k}B$ is noetherian. Then $A\hat{\otimes}_{k}B$ is a domain if and only if $%
 		A$ and $B$ are analytically irreducible.
 	\end{theorem}
 	
 	\begin{proof} The necessary condition follows from Corollary 3.3. We prove that the   condition is sufficient. We may suppose                                                     that $A$ and $B$ are complete.  By
 		Nagata Theorem, the integral
 		closure $A^{\prime }$ (resp. $B^{\prime }$) of $A$ (resp. $B$) is a finitely
 		generated $A$-module (resp. $B$-module) and a complete noetherian local
 		ring. We show at first that $A^{\prime }\hat{\otimes}_{k}B^{\prime }$ is a
 		domain. The ring $A^{\prime }$ (resp. $B^{\prime }$) dominates $A$ (resp. $B$)
 		then, if $\mathfrak{m}$ (resp. $\mathfrak{n}$) is the maximal ideal of $A$
 		(resp. $B$) and $\mathfrak{m}^{\prime }$ (resp. $\mathfrak{n}^{\prime }$) is
 		the maximal ideal of $A^{\prime }$ (resp. $B^{\prime }$)), the homomorphism $A/%
 		\mathfrak{m}\rightarrow A^{\prime }/\mathfrak{m}^{\prime }$ (resp. $B/%
 		\mathfrak{n}\rightarrow B^{\prime }/\mathfrak{n}^{\prime }$) is finite. We
 		deduce that the homomorphism {\footnotesize $(A/\mathfrak{m})\otimes _{k}(B/%
 			\mathfrak{n})\rightarrow (A^{\prime }/\mathfrak{m}^{\prime })\otimes
 			_{k}(B^{\prime }/\mathfrak{n}^{\prime })$ } is also finite. By Lemma 2.1
 		(iii), applied to $A\hat{\otimes}_{k}B$ and then to $A^{\prime }\hat{\otimes}%
 		_{k}B^{\prime }$, the ring $A^{\prime }\hat{\otimes}_{k}B^{\prime }$\ is
 		noetherian. In view of Lemma 5.6, it remains to show that $Spec((A^{\prime }/%
 		\mathfrak{m}^{\prime })\otimes _{k}(B^{\prime }/\mathfrak{n}^{\prime }))$ is
 		connected, which results from the fact that the ring $(A^{\prime }/\mathfrak{%
 			m}^{\prime })\otimes _{k}(B^{\prime }/\mathfrak{n}^{\prime })$ is a domain.
 		We show that $A\hat{\otimes}_{k}B$ is identified with a subring of $%
 		A^{\prime }\hat{\otimes}_{k}B^{\prime }$. By Theorem 3.1, the ring $A\hat{%
 			\otimes}_{k}B$ is identified with a subring of $%
 		A^{\prime }\otimes _{A}(A\hat{\otimes}_{k}B)$. As the $\mathfrak{m}$-preadic
 		topology of $A^{\prime }$ is also the $\mathfrak{m}^{\prime }$-preadic
 		topology, it follows from Lemma 5.7 that $A^{\prime }\otimes _{A}(A\hat{%
 			\otimes}_{k}B)$ is identified with the ring $A^{\prime }\hat{\otimes}_{k}B$.
 		Hence $A\hat{\otimes}_{k}B$ is identified with a subring of $A^{\prime }\hat{%
 			\otimes}_{k}B$. Since $A^{\prime }\hat{\otimes}_{k}B$ is noetherian, the
 		same reasoning shows that $A^{\prime }\hat{\otimes}_{k}B$ is identified with
 		a subring of $A^{\prime }\hat{\otimes}_{k}B^{\prime }$. We conclude that $A%
 		\hat{\otimes}_{k}B$ is a domain.
 	\end{proof}
 	
 	\begin{remarks}  \upshape

 		Let $\mathbf{R}$ be a property of a noetherian local ring and consider
 		for this property the following conditions:
 		\begin{itemize}
 			\item[$(\mathbf{R}_{\mathbf{0}%
 		}^{\prime})$]  Every regular local ring satisfies $\mathbf{%
 			R}$.
 		
 			\item[$(\mathbf{R}_{\mathbf{I}}^{\prime })$] For every flat local homomorphism $%
 		A\rightarrow B$\ of noetherian local with geometrically $\mathbf{R}$ fibers,
 		si $A$ satisfies $\mathbf{R}$ then $B$ satisfies $\mathbf{R}$.
 		
 			\item[$(\mathbf{R}_{\mathbf{II}})$] For every flat local homomorphism $A\rightarrow B$
 		\ of noetherian local rings, si $B$ satisfies $\mathbf{R}$ then $A$
 		satisfies $\mathbf{R}$.
 		
 			\item[$(\mathbf{R}_{\mathbf{IV}})$] For every noetherian local ring and for all non
 		zero divisor $t$ in its maximal ideal, if $A/tA$ satisfies $\mathbf{R}$ then $%
 		A$ satisfies $\mathbf{R}$.
 		
 			\item[$(\mathbf{R}_{\mathbf{V}})$] If a noetherian local $A$ satisfies $\mathbf{R}$,
 		then $A_{\mathfrak{p}}$ satisfies $\mathbf{R}$ for every prime ideal $%
 		\mathfrak{p}$ of $A$.
			\end{itemize}
		 
 	\begin{itemize}
 		\item[(i)] The same reasoning as in the proof of Theorem 5.4, where we replace Theorem 5.1 by Theorem B in \cite{Murayama}, show that we have the following result.

 		 {\bf Theorem.} {\it       Let $k$ be a perfect field, $A$ and $B$
 			two noetherian semi-local $k$-algebras such that the ring $A\hat{\otimes}%
 			_{k}B$ is noetherian. If \, $\mathbf{R}$\ satisfies $(\mathbf{R}_{\mathbf{0}%
 		}^{\prime })$, $(\mathbf{R}%
 			_{\mathbf{I}}^{\prime })$, $(\mathbf{R}_{\mathbf{II}})$, $(\mathbf{R}_{\mathbf{IV%
 			}})$, $(\mathbf{R}_{\mathbf{V}})$, then $A\hat{\otimes}_{k}B$ satisfies $\mathbf{%
 				R}$ if and only if $\widehat{A}$ and $\widehat{B}$ satisfies $\mathbf{R}$.}
 \item[(ii)]  The condition $\mathbf{R}_{\mathbf{I}}^{\prime }$ does not hold for the
property $\mathbf{R}$\ of being an domain (cf. \cite{Murayama}). 
 		
 	 \item[(iii)] The condition $\mathbf{R}_{\text{I}}^{\prime }$ does not hold
for the property $\mathbf{R}$\ of being an almost Cohen-Macaulay ring (cf.
\cite[Exemple 3.2]{Tabaa1}).

 		\end{itemize}	
\end{remarks}

 \end{document}